\documentclass[12pt]{article}
\usepackage{amsfonts}

\usepackage{graphicx}
\usepackage{amsmath}

\newtheorem{theorem}{Theorem}

\newtheorem{corollary}[theorem]{Corollary}

\newtheorem{definition}[theorem]{Definition}

\newtheorem{lemma}[theorem]{Lemma}

\newtheorem{remark}[theorem]{Remark}

\newenvironment{proof}[1][Proof]{\textbf{#1.} }{\ \rule{0.5em}{0.5em}}

\begin{document}

\title{{\LARGE Green kernel for a random Schr\"{o}dinger operator.
\footnote{This research was partially supported by the CONACYT.}}}
\author{  
Carlos G. Pacheco
\thanks{Departamento de Matematicas, CINVESTAV-IPN, A. Postal 14-740, Mexico D.F. 07000, MEXICO. Email: cpacheco@math.cinvestav.mx}}

\maketitle

\begin{abstract}
We find explicitly the Green kernel of a random Schr\"{o}dinger operator with Brownian white noise. 
To do this, we first handle the random operator by defining it weakly using the inner product of a Hilbert space.
Then, using classic Sturm-Liouville theory, we can build the Green kernel with linearly independent solutions of a homogeneous problem.
As a corollary we have that the random operator has a discrete spectra. 
\end{abstract}

\\



\section{Introduction}

In this paper we study a random Schr\"{o}dinger operator given symbolically by the expression
\begin{equation}\label{L}
[Lf](x):=-f^{\prime\prime}(x)+B^{\prime}(x)f(x),\ x\in [0,1], 
\end{equation}
acting on a class of functions with boundary conditions $f(0)=f(1)=0$, and where $\{B(x), x\in [0,1]\}$ represents the Brownian motion (BM). 
The object (\ref{L}) is an example of a random self-adjoint operator. 
And these operators have been important models in theoretical physics, in particular in the theory of disorder systems; standard references are \cite{Carmona, Pastur}.

Our main contribution is finding the Green operator associated to $L$, which helps, among other things, to study the spectra of $L$. 
Of course, since $B^{\prime}$ does not make sense, one has to define $L$ properly.
As a corollary of our result, we have that $L$ has almost always a discrete spectrum, result which was proved in \cite{Fuku} using theory of bilinear forms.
We take however a different route to have more detailed information of $L$.

Let us give an idea of our approach. 
As it is done in classic theory of Sturm-Liouville, the idea we carry out to find the spectra of $L$ is to give the associated Green operator, which is built using two linearly independent solutions of the homogeneous problem
\begin{equation}\label{EqHsol}
-f^{\prime\prime}+B^{\prime}f=0.
\end{equation} 
Of course, we have to say precisely what it means to solve such a homogeneous problem. 
Upon integrating (\ref{EqHsol}), one might propose studying the following integro-differential equation
\begin{equation}\label{EqIntDiffEq1}
u^{\prime}(1)-u^{\prime}(x)=\int_{x}^{1}u(y)dB(y).
\end{equation}
One problem with previous equation is that one does not have the value $u^{\prime}(1)$, or one does not even know if it actually exists.
To study (\ref{EqHsol}), instead of using (\ref{EqIntDiffEq1}), we use the inner product to define a weak solution. 
Nevertheless, our definition will eventually lead to (\ref{EqIntDiffEq1}).
As it is done in the classic case, the idea of constructing the associated Green operator is through two linearly independent solutions of (\ref{EqHsol}).

In the end, two good things help to find the spectra of $L$. 
The first one is that the spectra of $L$ is directly related to the spectra of the Green operator. 
The second thing is that the Green operator is almost surely a symmetric compact operator, and one can use the spectral theorem for such an operator.
Other information that one is able to extract is related to the bilinear form. The theory of bilinear forms assures the existence of a linear operator, call it $\tilde{L}$, that can be used to define $L$. Other consequence of the main result is that one is able to know the domain of $\tilde{L}$ as the image of the Green operator.

The paper is organized in the following way. In the coming Section \ref{SecDefL}, we make a rigorous definition of $L$. 
In Section \ref{SecBil} we mention the approach in \cite{Fuku} using bilinear forms.
Section \ref{SecGreen} has the main result which is the construction of the Green operator associated to $L$. 
This section also contains the proper meaning of solving (\ref{EqHsol}).
Section \ref{SecLema} has the proof of a lemma that helps to build the linearly independent solutions of the homogeneous problem (\ref{EqHsol}).
The conclusions of the paper are left in Section \ref{Conclusion}.

\section{Preliminaries}\label{SecDefL}
In this section we properly define the stochastic operator (\ref{L}). 
First of all, let us define the following linear spaces considered over the field $\mathbb{R}$:
\begin{equation*}
H:=\{h\in L_{2}[0,1]: h(0)=h(1)=0\}
\end{equation*}
and
\begin{equation*}
H_{1}:=\{h\in H : h\text{ is absolutely continuous }\}.
\end{equation*}

Now, one could define $L$ in the following weak sense
\begin{equation}\label{EqDef0L}
\langle Lf,g \rangle:= -\int_{0}^{1}f^{\prime\prime}(x)g(x)dx+\int_{0}^{1}f(x)g(x)dB(x),
\end{equation}
where $f,g\in H$ and $f$ is twice differentiable. 
Defining a random operator in a weak sense using an inner product is an important concept; this has recently been used in random matrix theory, see e.g. \cite{Ramirez}.
We will propose a definition inspired from \cite{Skorohod}.

Instead of using equation (\ref{EqDef0L}) directly, using integration by parts, it is more convenient for us to set the following.
\begin{definition}\label{DefL}
The stochastic Schr\"{o}dinger operator $L$ is such that for any pair $f,g\in H_{1}$, 
\begin{equation*}
\langle Lf,g \rangle:=\int_{0}^{1}f^{\prime}(x)g^{\prime}(x)dx+\int_{0}^{1}f(x)g(x)dB(x).
\end{equation*}
\end{definition}
Alternatively, using the It\^{o}'s formula, one has that
\begin{equation}\label{DefL2}
\langle Lf,g \rangle=\int_{0}^{1}f^{\prime}(x)g^{\prime}(x)dx-\int_{0}^{1}(f^{\prime}(x)g(x)+f(x)g^{\prime}(x))B(x)dx,
\end{equation}
which is precisely the form used in \cite{Fuku}.

An important issue is studying the spectra of operator $L$. Taking into account Definition \ref{DefL} we have the
\begin{definition}\label{DefEigen}
It is said that a value $\lambda\in \mathbb{C}$ is an eigenvalue of $L$, 
with eigenfunction $e\in H_{1}$, if 
\begin{equation*}
(\forall h\in H_{1})\langle Le, h \rangle= \lambda \langle e,h \rangle.
\end{equation*} 
\end{definition}

Finally, it should be mentined that there is a filtered probability space $(\Omega,\mathcal{F}, P)$ that supports the Brownian motion $B$.

\subsection{Using bilinear forms}\label{SecBil}
To prove that operator $L$ of Definition \ref{DefL} has a discrete spectra, in \cite{Fuku} they used theory of bilinear forms. 
In particular, they have taken the following approach.\\
i) Using (\ref{DefL2}) they work with the bilinear form: 
\begin{equation*}
\mathcal{E}(f,g):=\int_{0}^{1}f^{\prime}(x)g^{\prime}(x)dx-\int_{0}^{1}(f^{\prime}(x)g(x)+f(x)g^{\prime}(x))B(x)dx,
\end{equation*}
which is well defined for almost every trajectory of $B$.\\
ii) Using standard theory of bilinear forms, the authors argue that $\mathcal{E}$ admits the existence of an operator $\tilde{L}$ acting in a strong sense on a domain 
$\mathcal{D}\subset H_{1}$ (i.e. such that $\tilde{L}f$ is a well defined element in $H_{1}$ for each $f\in\mathcal{D}$), and such that 
\begin{equation*}
(\forall h\in H_{1}) \mathcal{E}(f,h)=\langle \tilde{L}f, h \rangle
\end{equation*} 
iii) After this, it is argued that $\tilde{L}$ has a point spectra.

\begin{remark}\label{NoteBilinear}
Regarding point ii), we would like see more precisely what $\tilde{L}$ and $\mathcal{D}$ are.
A priori (see Definition 10.4 in \cite{Schmudgen}), one knows that $\mathcal{D}$ are those $f\in H_{1}$ that admit a function $g\in H_{1}$ such that 
\begin{equation*}
(\forall h\in H_{1}) \mathcal{E}(f, h)=\langle g, h \rangle=\int_{0}^{1}g(x)h(x)dx.
\end{equation*}
\end{remark}
One can also revise these sort of ideas in \cite{Kato}, which was originally used in \cite{Fuku}.
Other related reference is \cite{Halperin}, also cited in \cite{Fuku}, where some heuristic arguments are proposed to carry out the analysis.

\section{The Green operator}\label{SecGreen}

We want now to construct the Green operator associated to $L$, but to do that we need the concept of solving the homogeneous equation (\ref{EqHsol}). 
Moreover, we need to say what it means to have two linearly independent solutions of this problem. So, we have the
\begin{definition}\label{DefIndSolHomo}
Consider the operator $L$ of Definition \ref{DefL}. 
We say that a stochastic process $\{u(x), x\in[0,1]\}$ with differentiable sample paths is a solution of the equation $Lf=0$ if
\begin{equation*}
(\forall h\in H_{1}) \langle Lu,h \rangle=0\text{ almost surely}.
\end{equation*}
Moreover, we say that two solutions $u$ and $v$ are linearly independent if the paths are almost always linearly independent functions. 
That is, for almost all $\omega\in \Omega$, the functions $u$ and $v$ are linearly independent.
\end{definition}

In the following result, we will show that there are indeed two solutions satisfying previous definition. 
In order to have linearly independence, we will construct $u$ such that $u(0)=0$ and $u(1)=1$ always, and on the other hand we will construct another solution $v$ such that $v(0)=1$ and $v(1)=0$ always. 
This guaranties the desired property.

\begin{lemma}\label{LemaL}
There there are two linearly independent solutions $u$ and $v$ of the problem $Lf=0$. 
In fact, $u$ and $v$ are respectively unique solutions of the following stochastic differential equations (SDE),
\begin{equation*}
u(x)-x=-\int_{0}^{1}K(x,y)u(y)dB(y)
\end{equation*} 
and
\begin{equation*}
v(x)+x-1=-\int_{0}^{1}K(x,y)v(y)dB(y),
\end{equation*}  
where
\begin{equation*}
K(x,y):=
\left\{
\begin{array}{ll}
x(1-y) &  0\leq x\leq y\leq 1\\
y(1-x) &  0\leq y\leq x\leq 1
\end{array}
\right.
\end{equation*}
Furthermore, almost surely 
\begin{equation}\label{Wronskian}
\alpha:=u^{\prime}(x)v(x)-u(x)v^{\prime}(x)\text{ is a positive constant for all }x\in [0,1].
\end{equation}
\end{lemma}
The proof of this lemma is in Section \ref{SecLema}. 

With the solutions $u$ and $v$ of Lemma \ref{LemaL}, we can build the following random operator, which turns out to be compact and symmetric almost surely.
\begin{definition}
Take $u$ and $v$ from Lemma \ref{LemaL}. 
The stochastic Green operator associated to $L$ is given by
\begin{equation}\label{defT}
[Tf](x):=\frac{1}{\alpha}\int_{0}^{1}G(x,y)f(y)dy,\ f\in H_{1},
\end{equation}
with the Green kernel
\begin{equation*}
G(x,y):=
\left\{
\begin{array}{ll}
u(x)v(y)  &  0\leq y\leq x\leq 1\\
u(y)v(x)  &  0\leq x\leq y\leq 1,
\end{array}
\right.
\end{equation*}
and $\alpha$ is the random variable (\ref{Wronskian}).
\end{definition}

We are now in position to prove that the Green operator (\ref{defT}) is the right inverse of $L$. 
Coincidently, the term ``right inverse'' may refer to position as well as correctness.
\begin{theorem}\label{TeoLT}
The operator $T$ in (\ref{defT}) is an inverse of $L$ from the right in the sense that
\begin{equation*}\label{EqInv}
(\forall h\in H_{1}) \langle LTf,h \rangle=\langle f,h \rangle.
\end{equation*}
\end{theorem}
It can be said as well that $LT$ is the identity for almost all $\omega\in \Omega$.

\begin{proof} Notice that $Tf$ in (\ref{defT}) can be written as 
\begin{equation*}
\alpha^{-1}v(x)\int_{0}^{x}u(y)f(y)dy+\alpha^{-1}u(x)\int_{x}^{1}v(y)f(y)dy.
\end{equation*}

We notice first from the specifications of $u$ and $v$ that $Tf\in H_{1}$.

Now, from previous display we have 
\begin{eqnarray*}
\alpha \langle LTf,h \rangle&=& 
\int_{0}^{1} v^{\prime}(x)\int_{0}^{x}u(y)f(y)dy h^{\prime}(x)dx+\int_{0}^{1} v(x)u(x)f(x) h^{\prime}(x)dx\\
&+&\int_{0}^{1} u^{\prime}(x)\int_{x}^{1}v(y)f(y)dy h^{\prime}(x)dx-\int_{0}^{1} u(x)v(x)f(x) h^{\prime}(x)dx\\
&+&\int_{0}^{1} v(x)\int_{0}^{x}u(y)f(y)dy h(x)dB(x)+\int_{0}^{1} u(x)\int_{x}^{1}v(y)f(y)dy h(x)dB(x)\\
&& (\text{now we add and subtract the following term})\\
&\pm& \int_{0}^{1}v^{\prime}(x)u(x)f(x)h(x)dx \pm  \int_{0}^{1}u^{\prime}(x)v(x)f(x)h(x)dx\\
&=&\int_{0}^{1}v^{\prime}(x)\left( h(x) \int_{0}^{x}u(y)f(y)dy \right)^{\prime}dx+\int_{0}^{1}v(x)\left( h(x)\int_{0}^{x}u(y)f(y)dy \right)dB(x)\\
&+&\int_{0}^{1}u^{\prime}(x)\left( h(x) \int_{x}^{1}v(y)f(y)dy \right)^{\prime}dx+\int_{0}^{1}u(x)\left( h(x)\int_{x}^{1}v(y)f(y)dy \right)dB(x)\\
&-&\int_{0}^{1}(v^{\prime}(x)u(x)-u^{\prime}(x)v(x))f(x)h(x)dx\\
&& \text{(using Lemma \ref{LemaL})}\\
&=&\alpha \int_{0}^{1}f(x)h(x)dx= \alpha \langle f,h \rangle.
\end{eqnarray*}
\end{proof}

\begin{remark}\label{NoteEigenTL}
We can now notice how the spectra of $T$ and $L$ are related to each other. 
Basically, suppose that $\beta$ is an eigenvalue of $T$ with eigenfunction $e\in H_{1}$, so that 
$Te=\beta e$, then, due to Theorem \ref{TeoLT}, it holds that 
\begin{equation*}
(\forall h\in H_{1})\langle e,h\rangle=\langle LTe,h \rangle=\beta \langle Le,h \rangle.
\end{equation*} 
Then, $1/\beta$ is an eigenvalue of $L$ with eigenfuntion $e$. 
\end{remark}

\begin{corollary}
The operator $L$ has almost surely a discrete spectrum given by the set of eigenvalues.
\end{corollary}
\begin{proof}
This is so because $T$ is almost surely a compact symmetric operator, then from the spectral theorem we know that there is almost surely a discrete spectrum of $T$ given by the eigenvalues. 
Finally, Remark \ref{NoteEigenTL} helps to conclude the proof of the statement.
\end{proof}

\begin{corollary}
With reference to Remark \ref{NoteBilinear}, the domain $\mathcal{D}$ of $\tilde{L}$ is given by $T(H_{1})$.
\end{corollary}
\begin{proof}
This happens because given $f\in T(H_{1})$, from Theorem \ref{TeoLT}, there exists a unique $g$ such that $f=T g$, and it holds that
\begin{equation*}
(\forall h\in H_{1}) \langle Lf,h \rangle=\langle g,h \rangle.
\end{equation*}
Thus $\tilde{L}f=g$.
\end{proof}


\section{Proof of Lemma \ref{LemaL}}\label{SecLema}

Let us first propose heuristically the independent solutions. 
Symbolically, we want to  see that there is $u$ such that it solves the following homogeneous boundary problem
\begin{equation*}
u^{\prime\prime}(x)=B^{\prime}(x)u(x)\ x\in [0,1],\ u(0)=0\ u(1)=1.
\end{equation*}
By a change of variable,
we can instead consider the equivalent formulation
\begin{equation*}
\overline{u}^{\prime\prime}(x)=B^{\prime}(x)(\overline{u}(x)+x)\ x\in [0,1],\ \overline{u}(0)=0\ \overline{u}(1)=0,
\end{equation*}
where $u(x)=\overline{u}(x)+x$. 
Now, appealing to Theorem 1 in \cite[p.215]{Cheney} we propose the following well-posed integral equation as an equivalent formulation to previous problem,
\begin{equation}\label{GreenEq}
\overline{u}(x)=-\int_{0}^{1}K(x,y)(\overline{u}(y)+y)dB(y),
\end{equation} 
where
\begin{equation*}
K(x,y):=
\left\{
\begin{array}{ll}
x(1-y) &  0\leq x\leq y\leq 1\\
y(1-x) &  0\leq y\leq x\leq 1
\end{array}
\right.
\end{equation*}

Equation (\ref{GreenEq}) is the same as
\begin{equation*}\label{EqSDEu}
u(x)-x=-\int_{0}^{1}K(x,y)u(y)dB(y),
\end{equation*} 
and from here we can establish the following SDE,
\begin{eqnarray}\label{EqIntEq}
x-u(x)&=&\int_{0}^{x}y(1-x)u(y)dB(y)+\int_{x}^{1}x(1-y)u(y)dB(y).
\end{eqnarray}
Since increments of $B$ taken in $[0,x]$ and $[x,1]$ are independent, 
we can use theory of SDEs with Lipschitz coefficients to prove that there is a unique solution. 

In a similar way we construct a stochastic process $v$ which is a candidate to solve
\begin{equation*}
v^{\prime\prime}(x)=B^{\prime}(x)v(x),\ x\in[0,1],\ v(0)=1,\ v(1)=0.
\end{equation*}
In this case, we need to solve the equation
\begin{equation*}
v(x)+x-1=-\int_{0}^{1}K(x,y)v(y)dB(y).
\end{equation*} 

\textbf{Claim.} The paths of $u$ and $v$ are differentiable functions.

\begin{proof}
Suppose for a moment that $u$ really admits a derivative $u^{\prime}$.
Using It\^{o}'s formula we have the following identities:
\begin{eqnarray*}
\int_{0}^{x}yu(y)dB(y)&=&-\int_{0}^{x}\{ yu^{\prime}(y)+u(y) \}B(y)dy+xu(x)B(x),\\
\int_{x}^{1}(1-y)u(y)dB(y)&=&-\int_{x}^{1}\{ (1-y)u^{\prime}(y)-u(y) \}B(y)dy-(1-x)u(x)B(x).
\end{eqnarray*}
After substituting previous formulas into (\ref{EqIntEq}) we have
\begin{eqnarray}\label{Equ}
x-u(x)&=& (x-1)\int_{0}^{x}\{ y u^{\prime}(y)+u(y) \}B(y)dy-x\int_{x}^{1}\{ (1-y)u^{\prime}(y)-u(y) \}B(y)dy.
\end{eqnarray}
We can take the derivative to arrive to
\begin{eqnarray}\label{EqIntDiffEq}
1-u^{\prime}(x)&=&(x-1)\{ xu^{\prime}(x)+u(x) \}B(x)+\int_{0}^{x}\{ yu^{\prime}(y)+u(y) \} B(y)dy \nonumber \\
&&+x\{ (1-x)u^{\prime}(x)-u(x)\}B(x)-\int_{x}^{1}\{(1-y)u^{\prime}(y)-u(y)\}B(y)dy\nonumber \\
&=&-u(x)B(x)-\int_{x}^{1}u^{\prime}(y)B(y)dy+c, 
\end{eqnarray}
where $c:=\int_{0}^{1}yu^{\prime}(y)B(y)dy+\int_{0}^{1}u(y)B(y)dy.$

If we set $h(x)=u^{\prime}(x)$, previous equation can be seen in the form
\begin{equation}\label{Eqh}
h(x)=f(x)+\int_{0}^{x}K_{1}(x,y)h(y)dy+\int_{x}^{1}K_{2}(x,y)h(y)dy,
\end{equation}
where $f, K_{1}$ and $K_{2}$ are known continuous functions. 
Therefore, one can check that the integral equation has a unique solution $h$, see e.g. Theorem 5 in \cite[pag. 183]{Cheney}.

We integrate equation (\ref{Eqh}) to go back to equation (\ref{Equ}). 
Thus, we end up the equation
\begin{eqnarray*}
x-\int_{0}^{x}h(y)dy&=& (x-1)\int_{0}^{x}\{ y h(y)+u(y) \}B(y)dy-x\int_{x}^{1}\{ (1-y)h(y)-u(y) \}B(y)dy.
\end{eqnarray*}
However, given $h(x)=u^{\prime}(x)$, 
the integral equation (\ref{Equ}) has also a unique solution $u$, therefore $u(x)=\int_{0}^{x}h(y)dy$.
This concludes the prove that $u$ is differentiable, and in a similar manner for $v$. 
\end{proof}

\bigskip

As a corollary, if we use apply  It\^{o}'s formula to (\ref{EqIntDiffEq}), we end up with the expression
\begin{equation}\label{EqIntDiffEq2}
u^{\prime}(x)= 1+\int_{0}^{1}yu(y)dB(y)-\int_{x}^{1}u(y)dB(y),
\end{equation}
which agrees with (\ref{EqIntDiffEq1}) and in fact gives the value of $u^{\prime}(1)$.

\bigskip
 
\begin{proof} (\textbf{of Lemma \ref{LemaL}}).
We now argue why process $u$ coming from solving the SDE (\ref{EqIntEq}) satisfies Definition \ref{DefIndSolHomo}. 
Consider a sequence $\{B_{n}(x), x\in[0,1]\}_{n=1}^{\infty}$ of piecewise approximations of $B$ in the sense of \cite[p.505]{Ikeda}:
\begin{equation*}
B_{n}(x):=n\left[ \left( \frac{j+1}{n}-x \right)B(j/n)+\left(x-\frac{j}{n} \right)B((j+1)/n)\right],
\end{equation*}
whenever $x\in [(j+1)/n, j/n]$, $j=0,1,2,\ldots$

Let $u_{n}$ be the sequence of solutions of the SDE (\ref{EqIntEq}) substituting $B$ by $B_{n}$. 
Then, each $u_{n}$ satisfies Definition \ref{DefIndSolHomo}.

Now, we can make use of the results in \cite{Ikeda} of approximation of SDEs.
First, since $u$ is a differentiable function, its It\^{o}'s integral coincides with its Stratonovich integral.
Furthermore, a quadratic covariation involving a continuous function times $u$ will be zero. 
Then, from the corollary of Theorem 7.3 in \cite[p.516]{Ikeda}, we can take a subsequence of $u_{n}$ that converges uniformly on $[0,1]$ to $u$ almost surely, abusing of the notation, let us call such subsequence again $u_{n}$.
We also see from (\ref{EqIntDiffEq2}) that $u_{n}^{\prime}(x)$ converges almost surely to $u^{\prime}(x)$ for each $x\in[0,1]$;  however, since $[0,1]$ is compact and the functions are continuous, the convergence of $u^{\prime}_{n}$ is also uniform on the whole interval $[0,1]$.
Therefore, using this uniform convergence and the fact that $u_{n}$ satisfies Definition \ref{DefIndSolHomo}, we have that for any $h \in H_{1}$ almost surely
\begin{equation*}
\langle Lu,h \rangle=\lim_{n\to\infty}\langle Lu_{n},h \rangle
=\lim_{n\to\infty}\left\{ \int_{0}^{1}u_{n}^{\prime}(x)h^{\prime}(x)dx+\int_{0}^{1}u_{n}(x)h(x)dB(x)\right\}=0.
\end{equation*} 

We can now conclude that $u$ satisfies Definition \ref{DefIndSolHomo}, and the same for $v$.

In conclusion, $u$ and $v$ are the two linearly independent solutions of the homogeneous problem such that $u(0)=v(1)=0$. 
We now want to prove (\ref{Wronskian}), but from the theory of Sturm-Liouville we also know that almost surely (see e.g. \cite[pag. 107]{Cheney})
\begin{equation*}\label{Wronskiann}
u^{\prime}_{n}(x)v_{n}(x)-u_{n}(x)v^{\prime}_{n}(x)\text{ is constant for all }x\in [0,1],
\end{equation*}
for each $n=1,2,\ldots$. 
With the approximating result mentioned above, we can finally establish (\ref{Wronskian}), and therefore Lemma \ref{LemaL} has been proved. 
\end{proof}


\section{Conclusions}\label{Conclusion}
In this paper we have studied a Schr\"{o}dinger operator with a white noise potential. 
We made use of the concept of a weak random operator and standard theory of Sturm-Liouville to propose a framework that helps to study the spectra.
In particular we were able to find the Green kernel that helps to invert the random Schr\"{o}dinger operator.

\bigskip

\textbf{Acknowledgements.} 
The author wishes to thank Nikolai Vasilevski for the fruitful discussions on bilinear forms, and Vladislav Kravchenko for few comments on solving the equation. 
We also thank the anonymous referee for helping to improve the paper.


\end{document}